\documentclass[numbers=enddot,12pt,final,onecolumn,notitlepage]{scrartcl}%
\usepackage[headsepline,footsepline,manualmark]{scrlayer-scrpage}
\usepackage{amsfonts}
\usepackage{amssymb}
\usepackage{amsmath}
\usepackage{amsthm}
\usepackage{framed}
\usepackage{comment}
\usepackage{color}
\usepackage[breaklinks=true]{hyperref}
\usepackage[sc]{mathpazo}
\usepackage[T1]{fontenc}
\usepackage{color}
\usepackage{needspace}
\usepackage{tabls}
\providecommand{\U}[1]{\protect\rule{.1in}{.1in}}
\theoremstyle{definition}
\newtheorem{theo}{Theorem}[section]
\newenvironment{theorem}[1][]
{\begin{theo}[#1]\begin{leftbar}}
{\end{leftbar}\end{theo}}
\newtheorem{lem}[theo]{Lemma}
\newenvironment{lemma}[1][]
{\begin{lem}[#1]\begin{leftbar}}
{\end{leftbar}\end{lem}}
\newtheorem{prop}[theo]{Proposition}
\newenvironment{proposition}[1][]
{\begin{prop}[#1]\begin{leftbar}}
{\end{leftbar}\end{prop}}
\newtheorem{defi}[theo]{Definition}

\newtheorem{remk}[theo]{Remark}

\newtheorem{coro}[theo]{Corollary}

\newtheorem{conv}[theo]{Convention}

\newtheorem{quest}[theo]{Question}
\newenvironment{question}[1][]
{\begin{quest}[#1]\begin{leftbar}}
{\end{leftbar}\end{quest}}
\newtheorem{warn}[theo]{Warning}

\newtheorem{conj}[theo]{Conjecture}

\newtheorem{exam}[theo]{Example}

\newenvironment{statement}{\begin{quote}}{\end{quote}}

\newcommand{\arinj}{\ar@{_{(}->}}
\newcommand{\arinjrev}{\ar@{^{(}->}}
\newcommand{\arsurj}{\ar@{->>}}
\newcommand{\arelem}{\ar@{|->}}
\newcommand{\arback}{\ar@{<-}}

\let\sumnonlimits\sum
\let\prodnonlimits\prod
\let\cupnonlimits\bigcup
\let\capnonlimits\bigcap
\renewcommand{\sum}{\sumnonlimits\limits}
\renewcommand{\prod}{\prodnonlimits\limits}
\renewcommand{\bigcup}{\cupnonlimits\limits}
\renewcommand{\bigcap}{\capnonlimits\limits}

\newenvironment{noncompile}{}{}
\excludecomment{noncompile}
\ihead{Triple parabolic cosets
}\ohead{page \thepage}
\begin{document}

\title{Triple parabolic cosets have unique Bruhat minima}
\date{\today}
\author{Mikey Becht, Darij Grinberg}
\maketitle

\begin{abstract}
\textbf{Abstract.} A folklore result says that any parabolic double coset
$W_{I} w W_{J}$ in a Coxeter group $W$ is an interval of the Bruhat order,
provided that $W_{I}$ and $W_{J}$ are finite. We include a self-contained
proof, without assuming that $W$ is finite. Furthermore, we prove that each
``triple coset'' $W_{I} u W_{J} v W_{K}$ has a Bruhat minimum and -- if
$W_{I}, W_{J}, W_{K}$ are finite -- a Bruhat maximum, but (in general) is not
a Bruhat interval. This proves and extends a conjecture on MathOverflow \#472072.

\end{abstract}

\section{Reminders on the strong Bruhat order}

Let $\left(  W,S\right)  $ be a Coxeter system. We refer to standard texts
such as \cite{BjoBre}, \cite{Davis}, \cite{Schremmer}, \cite{Lusztig},
\cite{GecPfi00} for definitions and properties of Coxeter groups. The length
of an element $w\in W$ will be denoted by $\ell\left(  w\right)  $.

We shall use the \emph{strong Bruhat order} on $W$. This is the partial order
on $W$ that is called the \textquotedblleft Bruhat order\textquotedblright\ in
\cite{BjoBre}, and that is defined but left nameless in \cite[\S 2.1]%
{Lusztig}. One of the many ways to define it is as follows:

\begin{proposition}
\label{prop.bruhat-redex2}Let $u,v\in W$ be two elements. Let $s_{1}%
s_{2}\cdots s_{m}$ be a reduced expression for $v$. Then, $u\leq v$ (in the
strong Bruhat order) if and only if $u$ can be written as a subproduct\ of
this reduced expression (i.e., if there exist $i_{1}<i_{2}<\cdots<i_{k}$ in
$\left\{  1,2,\ldots,m\right\}  $ such that $u=s_{i_{1}}s_{i_{2}}\cdots
s_{i_{k}}$).
\end{proposition}

This indeed can be used as a definition of $\leq$ (that is, the meaning of
$u\leq v$ is independent of the choice of the reduced expression) and is
equivalent to the other definitions of the strong Bruhat order; see
\cite[Proposition 2.4 (iii)]{Lusztig} or \cite[Theorem 2.2.2]{BjoBre} or
\cite[\S 5.10]{Humphreys} for proofs.

Henceforth, inequality signs such as $\leq$, $<$, $\geq$ and $>$ will be
understood as referring to the strong Bruhat order whenever they appear
between elements of $W$. We recall some properties of the strong Bruhat order:

\begin{proposition}
\label{prop.bruhat.props1}Let $u,v\in W$. Then:

\begin{enumerate}
\item[\textbf{(a)}] If $u\leq v$, then $\ell\left(  u\right)  \leq\ell\left(
v\right)  $.

\item[\textbf{(b)}] If $u<v$, then $\ell\left(  u\right)  <\ell\left(
v\right)  $.

\item[\textbf{(c)}] We have $u\leq v$ if and only if $u^{-1}\leq v^{-1}$.

\item[\textbf{(d)}] We have $u<v$ if and only if $u^{-1}<v^{-1}$.
\end{enumerate}
\end{proposition}

\begin{proof}
All four parts follow easily from Proposition \ref{prop.bruhat-redex2}. (For
parts \textbf{(a)} and \textbf{(b)}, use the trivial observation that
$\ell\left(  s_{i_{1}}s_{i_{2}}\cdots s_{i_{k}}\right)  \leq k$ for any
$i_{1}<i_{2}<\cdots<i_{k}$. For parts \textbf{(c)} and \textbf{(d)}, observe
that if $s_{1}s_{2}\cdots s_{m}$ is a reduced expression for some $w\in W$,
then $s_{m}s_{m-1}\cdots s_{1}$ is a reduced expression for $w^{-1}$.)
\end{proof}

For the next proposition, we recall the notion of a reflection. The elements
of $S$ are called the \emph{simple reflections} (of $W$). The elements of $W$
that are conjugate to simple reflections are called the \emph{reflections} (of
$W$). Commonly, the set of all reflections is denoted by $T$ (although we have
no need for this notation).

\begin{proposition}
\label{prop.bruhat.reflec} Let $t\in W$ be a reflection. Let $w\in W$ be any
element. Then:

\begin{enumerate}
\item[\textbf{(a)}] If $\ell\left(  w\right)  <\ell\left(  tw\right)  $, then
$w<tw$.

\item[\textbf{(b)}] If $\ell\left(  w\right)  <\ell\left(  wt\right)  $, then
$w<wt$.
\end{enumerate}
\end{proposition}

\begin{proof}
\textbf{(a)} This follows immediately from the definition of the strong Bruhat
order given in \cite[Definition 2.1.1]{BjoBre}. Indeed, $w^{-1}tw$ is
conjugate to the reflection $t$, and thus itself a reflection; hence,
$\ell\left(  w\right)  <\ell\left(  tw\right)  $ entails $w\overset{w^{-1}%
tw}{\rightarrow}tw$ in the language of \cite[Definition 2.1.1]{BjoBre}, and
therefore $w<tw$. This proves Proposition \ref{prop.bruhat.reflec}
\textbf{(a)}.

Alternatively, Proposition \ref{prop.bruhat.reflec} \textbf{(a)} can also be
proved using the definition of the strong Bruhat order given in \cite[\S 2.1]%
{Lusztig}. Namely, we observe that $t^{2}=1$ (since any reflection is an
involution). Set $u:=tw$, so that $tu=\underbrace{tt}_{=t^{2}=1}w=w$. Pick any
reduced expression $u=s_{1}s_{2}\cdots s_{q}$ of $u$. Assume that $\ell\left(
w\right)  <\ell\left(  tw\right)  $. Since $tu=w$, we now have $\ell\left(
tu\right)  =\ell\left(  w\right)  <\ell\left(  tw\right)  =\ell\left(
u\right)  $ (since $tw=u$). Hence, \cite[Lemma 2.2, (iii) $\Longrightarrow$
(ii)]{Lusztig} (applied to $u$ instead of $w$) yields that $t=s_{1}s_{2}\cdots
s_{i}\cdots s_{2}s_{1}$ for some $i\in\left\{  1,2,\ldots,q\right\}  $.
Consider this $i$. Now,%
\begin{align*}
w  &  =\underbrace{t}_{\substack{=s_{1}s_{2}\cdots s_{i}\cdots s_{2}%
s_{1}\\=\left(  s_{1}s_{2}\cdots s_{i-1}\right)  \left(  s_{i}s_{i-1}\cdots
s_{1}\right)  }}\underbrace{u}_{\substack{=s_{1}s_{2}\cdots s_{q}\\=\left(
s_{1}s_{2}\cdots s_{i}\right)  \left(  s_{i+1}s_{i+2}\cdots s_{q}\right)  }}\\
&  =\left(  s_{1}s_{2}\cdots s_{i-1}\right)  \underbrace{\left(  s_{i}%
s_{i-1}\cdots s_{1}\right)  \cdot\left(  s_{1}s_{2}\cdots s_{i}\right)
}_{\substack{=1\\\text{(since }s_{1},s_{2},\ldots,s_{i}\text{ are
involutions)}}}\left(  s_{i+1}s_{i+2}\cdots s_{q}\right) \\
&  =s_{1}s_{2}\cdots s_{i-1}s_{i+1}s_{i+2}\cdots s_{q}.
\end{align*}
Since this is clearly a subproduct of the reduced expression $s_{1}s_{2}\cdots
s_{q}$ of $u$, we thus conclude (by Proposition \ref{prop.bruhat-redex2}) that
$w\leq u$. Thus, $w\leq u=tw$. Since $w\neq tw$ (because $t\neq1$), we thus
obtain $w<tw$. This proves Proposition \ref{prop.bruhat.reflec} \textbf{(a)}
again. \medskip

\textbf{(b)} This can be easily derived from part \textbf{(a)} using
Proposition \ref{prop.bruhat.props1} \textbf{(d)}. Namely: Assume that
$\ell\left(  w\right)  <\ell\left(  wt\right)  $. It is well-known (and fairly
clear) that any $u\in W$ satisfies $\ell\left(  u^{-1}\right)  =\ell\left(
u\right)  $. Hence, $\ell\left(  w^{-1}\right)  =\ell\left(  w\right)  $ and
$\ell\left(  \left(  wt\right)  ^{-1}\right)  =\ell\left(  wt\right)  $.
However, $t$ is a reflection, thus an involution; hence, $t^{-1}=t$. Thus,
$\left(  wt\right)  ^{-1}=\underbrace{t^{-1}}_{=t}w^{-1}=tw^{-1}$. Thus, we
can rewrite $\ell\left(  \left(  wt\right)  ^{-1}\right)  =\ell\left(
wt\right)  $ as $\ell\left(  tw^{-1}\right)  =\ell\left(  wt\right)  $. Now,
we have%
\[
\ell\left(  w^{-1}\right)  =\ell\left(  w\right)  <\ell\left(  wt\right)
=\ell\left(  tw^{-1}\right)  .
\]
Thus, Proposition \ref{prop.bruhat.reflec} \textbf{(a)} (applied to $w^{-1}$
instead of $w$) yields $w^{-1}<tw^{-1}=\left(  wt\right)  ^{-1}$. Therefore,
Proposition \ref{prop.bruhat.props1} \textbf{(d)} (applied to $u=w$ and
$v=wt$) yields $w<wt$. This proves Proposition \ref{prop.bruhat.reflec}
\textbf{(b)}.
\end{proof}

The last property of the Bruhat order that we will need is a known lemma
\cite[Lemma 2.2.10]{BjoBre}:

\begin{lemma}
\label{lem.bjobre22}Let $x,y\in W$ be arbitrary, and let $t\in W$ be any
reflection. Assume that $x<xt$ and $y<ty$. Then, $xy<xty$.
\end{lemma}

\begin{proof}
We have $x<xt$, thus $\ell\left(  x\right)  < \ell\left(  xt\right)  $ by
Proposition \ref{prop.bruhat.props1} \textbf{(b)}. That is, $\ell\left(
xt\right)  > \ell\left(  x\right)  $. Likewise, from $y<ty$, we obtain
$\ell\left(  ty\right)  > \ell\left(  y\right)  $.

Set $t^{\prime}:=xtx^{-1}$. Then, $xt=t^{\prime}x$, so that $xty=t^{\prime}%
xy$. Hence, we must show that $xy<t^{\prime}xy$. Since $t^{\prime}$ is a
reflection (because it is conjugate to the reflection $t$), this will follow
from Proposition \ref{prop.bruhat.reflec} \textbf{(a)} (applied to $xy$ and
$t^{\prime}$ instead of $w$ and $t$) once we can show that $\ell\left(
xy\right)  <\ell\left(  t^{\prime}xy\right)  $. But \cite[errata, Lemma
4.2.4]{Davis} shows that $\ell\left(  xty\right)  >\ell\left(  xy\right)  $.
Hence, $\ell\left(  xy\right)  <\ell\left(  xty\right)  =\ell\left(
t^{\prime}xy\right)  $ (since $xt=t^{\prime}x$), just as we desired to prove.
\end{proof}

For any $u,v\in W$ satisfying $u\leq v$, we define $\left[  u,v\right]  $ to
be the set%
\[
\left\{  x\in W\ \mid\ u\leq x\leq v\right\}  .
\]
This set is called a \emph{Bruhat interval}.

If $X$ is a subset of $W$, then an element $m\in X$ will be called a
\emph{Bruhat minimum} of $X$ if all $x\in X$ satisfy $m\leq x$. Not every
subset $X$ of $W$ has a Bruhat minimum. However, when the Bruhat minimum of a
subset $X$ does exist, then it is unique (since two Bruhat minima $m_{1}$ and
$m_{2}$ of $X$ would have to satisfy $m_{1}\leq m_{2}$ and $m_{2}\leq m_{1}$
and therefore $m_{1}=m_{2}$) and is the unique element of $X$ having minimum
length (because if $m$ is a Bruhat minimum of $X$, then every $x\neq m$ in $X$
satisfies $m<x$ and therefore $\ell\left(  m\right)  <\ell\left(  x\right)  $
by Proposition \ref{prop.bruhat.props1} \textbf{(b)}).

All the same applies to Bruhat maxima: If $X$ is a subset of $W$, then an
element $M\in X$ will be called a \emph{Bruhat maximum} of $X$ if all $x\in X$
satisfy $M\geq x$. Again, if such an $M$ exists, then it is unique and is the
unique element of $X$ having maximum length.

Any Bruhat interval $\left[  u,v\right]  $ has the Bruhat minimum $u$ and the
Bruhat maximum $v$. However, not every subset of $W$ that has a Bruhat minimum
and a Bruhat maximum is a Bruhat interval.

\section{The theorems}

If $I$ is any subset of $S$, then we let $W_{I}$ denote the subgroup of $W$
generated by $I$. It is well-known (see \cite[Proposition 2.4.1 (i) and
(ii)]{BjoBre} or \cite[\S 9.4 and \S 9.6]{Lusztig}) that $W_{I}$ is itself a
Coxeter group (with simple reflections $I$) whose length function is a
restriction of that of $W$. It is called a \emph{parabolic subgroup} of $W$
(although some authors, such as Humphreys in \cite{Humphreys}, use this word
more widely for conjugates of such subgroups).

Given two subsets $I$ and $J$ of $S$, the double cosets $W_{I}wW_{J}$ in $W$
are called \emph{parabolic double cosets}; they have been studied by various
authors from the viewpoints of descent algebras (\cite{Solomo76}), Hecke
algebras (\cite{Curtis}), row-semistandard tableaux (\cite{Wildon14}) and
contingency tables (\cite{HohSka05}). Recently, the total number of distinct
parabolic double cosets in a finite Coxeter group was computed (\cite{BKPST18}).

A classical result (essentially a consequence of \cite[Exercise IV.\S 1.3]%
{Bourbaki-Lie4-6}) says that each parabolic double coset $W_{I}wW_{J}$, where
$I$ and $J$ are two subsets of $S$ (and where $W_{I}$ and $W_{J}$ are the
corresponding parabolic subgroups of $W$), has a Bruhat minimum, called the
\textquotedblleft distinguished representative\textquotedblright\ of this
double coset. If $W_{I}$ and $W_{J}$ are finite, then each parabolic double
coset $W_{I}wW_{J}$ also has a Bruhat maximum, and is an interval in the
Bruhat order. This appears to be a folklore result (see, e.g.,
\cite[Proposition 2.7 (c)]{BKPST18} and \cite{Kobaya11}), but proofs are not
easy to find in the literature, so we will include a proof for the sake of
completeness:

\begin{theorem}
\label{thm.double-interval} Let $I$ and $J$ be two subsets of $S$ such that
the parabolic subgroups $W_{I}$ and $W_{J}$ are finite. Let $w\in W$. Then,
the double coset $W_{I}wW_{J}$ is a Bruhat interval.
\end{theorem}

In 2024, the MathOverflow user \textquotedblleft fool rabbit\textquotedblright%
\ conjectured\footnote{in \url{https://mathoverflow.net/questions/472072/}}
that \textquotedblleft parabolic triple cosets\textquotedblright\ -- that is,
subsets of the form $W_{I}uW_{J}vW_{K}$ for three subsets $I,J,K$ of $S$ --
also have Bruhat minima and maxima when $W$ is finite. Our main goal in this
note is to prove this conjecture, and even dispense with the finiteness
condition to the extent possible\footnote{We cannot hope to avoid requiring
finiteness of $W_{I},W_{J},W_{K}$ in Theorem \ref{thm.triplemax}, since an
infinite subset of $W$ cannot have a Bruhat maximum (after all, for any $w\in
W$, there are only finitely many elements $u\in W$ satisfying $u\leq w$).}:

\begin{theorem}
\label{thm.triple} Let $I,J,K$ be three subsets of $S$. Consider the parabolic
subgroups $W_{I},W_{J},W_{K}$. Let $u,v\in W$ be arbitrary. Then, the subset
$W_{I}uW_{J}vW_{K}$ of $W$ has a Bruhat minimum.
\end{theorem}

\begin{theorem}
\label{thm.triplemax} Let $I,J,K$ be three subsets of $S$. Assume that the
parabolic subgroups $W_{I},W_{J},W_{K}$ are finite. Let $u,v\in W$ be
arbitrary. Then, the subset $W_{I}uW_{J}vW_{K}$ of $W$ has a Bruhat maximum.
\end{theorem}

However, triple cosets are not as well-behaved as double cosets. In
particular, even if $W$ is finite, they are not always Bruhat intervals
(unlike double cosets, for which we know this from Theorem
\ref{thm.double-interval}). For example, in the symmetric group $S_{3}$
(regarded as a Coxeter group of type $A_{2}$), the triple coset
$W_{\varnothing}s_{2}W_{\left\{  s_{1}\right\}  }s_{2}W_{\varnothing}%
=s_{2}W_{\left\{  s_{1}\right\}  }s_{2}=\left\{  1,w_{0}\right\}  $ is not a
Bruhat interval.

One could also define \textquotedblleft parabolic quadruple
cosets\textquotedblright\ $W_{I}uW_{J}vW_{K}wW_{L}$. However, these do not
satisfy an analogue of Theorem \ref{thm.triple}. For instance, if $W$ is the
symmetric group $S_{3}$, then the \textquotedblleft quadruple
coset\textquotedblright\ $W_{\varnothing}s_{1}W_{\left\{  s_{2}\right\}
}1W_{\left\{  s_{1}\right\}  }s_{2}W_{\varnothing}$ has two distinct elements
of minimum length ($s_{1}$ and $s_{2}$), and so has no Bruhat minimum.

\section{Double cosets revisited}

We will use some classical properties of double cosets in Coxeter groups:

\begin{proposition}
\label{prop.double-coset}Let $P$ and $Q$ be two subsets of $S$. Let $w\in W$. Then:

\begin{enumerate}
\item[\textbf{(a)}] The double coset $W_{P}wW_{Q}$ has a unique element of
minimum length.
\end{enumerate}

Let us denote this element by $m$. Then:

\begin{enumerate}
\item[\textbf{(b)}] Each $u\in W_{P}wW_{Q}$ can be written as $u=xmy$ with
$x\in W_{P}$ and $y\in W_{Q}$.

\item[\textbf{(c)}] The element $m$ is the Bruhat minimum of $W_{P}wW_{Q}$.

\item[\textbf{(d)}] Let $x\in W_{P}$ and $y\in W_{Q}$ be such that
$\ell\left(  xmy\right)  \neq\ell\left(  x\right)  +\ell\left(  m\right)
+\ell\left(  y\right)  $. Then, there exist $x^{\prime}\in W_{P}$ and
$y^{\prime}\in W_{Q}$ such that $x^{\prime}<x$ and $y^{\prime}<y$ and
$xmy=x^{\prime}my^{\prime}$.
\end{enumerate}
\end{proposition}

Most of this proposition is well-known (see, e.g., \cite{Curtis} and
\cite[Proposition 8.3]{GarSta}), and in particular its part \textbf{(a)}
appears in \cite[Proposition 9.15 \textbf{(a)}]{Lusztig}. However, we shall
give a self-contained proof using nothing but the above definition of the
Bruhat order and the well-known \emph{deletion property} of Coxeter groups:

\begin{theorem}
[Deletion property]\label{thm.del} Let $w \in W$ and $s_{1}, s_{2}, \ldots,
s_{k} \in S$. Assume that $w = s_{1} s_{2} \cdots s_{k}$, but this is
\textbf{not} a reduced expression for $w$ (that is, we have $\ell\left(
w\right)  < k$). Then, we can delete two factors from the product $s_{1} s_{2}
\cdots s_{k}$ in such a way that the resulting product is still $w$. (That is,
there exist two numbers $i < j$ in $\left\{  1,2,\ldots,k \right\}  $ such
that $w = s_{1} s_{2} \cdots\widehat{s_{i}} \cdots\widehat{s_{j}} \cdots
s_{k}$, where the hat over a factor means that this factor is excluded from
the product.)
\end{theorem}

\begin{proof}
See \cite[Proposition 1.4.7]{BjoBre} or \cite[Theorem 3.3.4 (iii) + Theorem
3.2.16 (E) $\Longrightarrow$ (D)]{Davis} or \cite[(a) $\Longrightarrow$
(d)]{Schremmer} or \cite[Theorem 1.2.5]{GecPfi00} or \cite[Corollary
1.23]{Sam21} or \cite[Satz 2.15]{Werner12}.
\end{proof}

\begin{proof}
[Proof of Proposition \ref{prop.double-coset}.]We shall postpone the proof of
part \textbf{(a)} for now. Of course, this complicates our life, since the
uniqueness of the element $m$ is no longer clear. Thus, instead of defining
$m$ to be the unique element of minimum length in $W_{P}wW_{Q}$, let us merely
require that $m$ be \textbf{an} element of minimum length in $W_{P}wW_{Q}$. We
shall now prove parts \textbf{(b)}, \textbf{(c)} and \textbf{(d)} of
Proposition \ref{prop.double-coset} under this (weaker-sounding) condition.
Afterwards, we will derive part \textbf{(a)} from part \textbf{(c)}. \medskip

\textbf{(b)} Let $u\in W_{P}wW_{Q}$. From $m\in W_{P}wW_{Q}$, we obtain
$W_{P}wW_{Q}=W_{P}mW_{Q}$. Hence, $u\in W_{P}wW_{Q}=W_{P}mW_{Q}$. In other
words, there exists a factorization $u=xmy$ of $u$ with $x\in W_{P}$ and $y\in
W_{Q}$. This proves part \textbf{(b)}. \medskip

\textbf{(d)} In the following, the \textquotedblleft hat\textquotedblright%
\ symbol ($\,\widehat{\ \ }\, $) over a factor in a product signifies that
this factor must be omitted. For instance, if $a_{1},a_{2},\ldots,a_{k}$ are
some elements of a group, and if $i,j\in\left\{  1,2,\ldots,k\right\}  $ are
two indices with $i<j$, then \textquotedblleft$a_{1}a_{2}\cdots\widehat{a_{i}%
}\cdots\widehat{a_{j}}\cdots a_{k}$\textquotedblright\ means the product%
\[
a_{1}a_{2}\cdots a_{i-1}a_{i+1}a_{i+2}\cdots a_{j-1}a_{j+1}a_{j+2}\cdots a_{k}%
\]
(that is, the product $a_{1}a_{2}\cdots a_{k}$ with the two factors $a_{i}$
and $a_{j}$ omitted).

Let $p_{1}p_{2}\cdots p_{i}$ and $q_{1}q_{2}\cdots q_{j}$ and $m_{1}%
m_{2}\cdots m_{\ell}$ be reduced expressions for $x$ and $y$ and $m$,
respectively, where all the $p_{k}$ belong to $P$ (since $x\in W_{P}$), and
where all the $q_{k}$ belong to $Q$ (since $y\in W_{Q}$). Thus, $\ell\left(
x\right)  =i$ and $\ell\left(  y\right)  =j$ and $\ell\left(  m\right)  =\ell$
and $x=p_{1}p_{2}\cdots p_{i}$ and $y=q_{1}q_{2}\cdots q_{j}$ and
$m=m_{1}m_{2}\cdots m_{\ell}$. The latter three equalities yield%
\begin{equation}
xmy=p_{1}p_{2}\cdots p_{i}m_{1}m_{2}\cdots m_{\ell}q_{1}q_{2}\cdots q_{j}.
\label{pf.prop.double-coset.red?}%
\end{equation}
This \textbf{cannot} be a reduced expression for $xmy$, because if it were,
then we would obtain%
\[
\ell\left(  xmy\right)  =i+\ell+j=\ell\left(  x\right)  +\ell\left(  m\right)
+\ell\left(  y\right)  ,
\]
which would contradict our condition $\ell\left(  xmy\right)  \neq\ell\left(
x\right)  +\ell\left(  m\right)  +\ell\left(  y\right)  $.

So the product $p_{1}p_{2}\cdots p_{i}m_{1}m_{2}\cdots m_{\ell}q_{1}%
q_{2}\cdots q_{j}$ is not a reduced expression. Hence, by the deletion
property (Theorem~\ref{thm.del}), we can remove two of the factors
\[
p_{1},p_{2},\ldots,p_{i},m_{1},m_{2},\ldots,m_{\ell},q_{1},q_{2},\ldots,q_{j}
\]
from this product without changing the value of this product. Let us refer to
these two factors as the two \emph{redundant factors} (since their removal
leaves the product unchanged). We are in one of the following five cases:

\textit{Case 1:} Both redundant factors have the forms $m_{a}$ and $m_{b}$ for
some $a,b\in\left\{  1,2,\ldots,\ell\right\}  $.

\textit{Case 2:} Both redundant factors have the forms $p_{a}$ and $p_{b}$ for
some $a,b\in\left\{  1,2,\ldots,i\right\}  $.

\textit{Case 3:} Both redundant factors have the forms $q_{a}$ and $q_{b}$ for
some $a,b\in\left\{  1,2,\ldots,j\right\}  $.

\textit{Case 4:} Only \textbf{one} of the redundant factors has the form
$m_{a}$ for some $a\in\left\{  1,2,\ldots,\ell\right\}  $ (while the other has
the form $p_{b}$ or $q_{b}$).

\textit{Case 5:} One of the redundant factors has the form $p_{a}$ for some
$a\in\left\{  1,2,\ldots,i\right\}  $, while the other has the form $q_{b}$
for some $b\in\left\{  1,2,\ldots,j\right\}  $.

We shall now show that Cases 1, 2, 3 and 4 are all impossible. Indeed:

\begin{itemize}
\item Assume that we are in Case 1. Thus, both redundant factors have the
forms $m_{a}$ and $m_{b}$ for some $a,b\in\left\{  1,2,\ldots,\ell\right\}  $.
Consider these $a,b$. WLOG assume that $a<b$. Then,%
\begin{align*}
&  p_{1}p_{2}\cdots p_{i}m_{1}m_{2}\cdots m_{\ell}q_{1}q_{2}\cdots q_{j}\\
&  =p_{1}p_{2}\cdots p_{i}m_{1}m_{2}\cdots\widehat{m_{a}}\cdots\widehat{m_{b}%
}\cdots m_{\ell}q_{1}q_{2}\cdots q_{j}%
\end{align*}
(since the removal of the redundant factors leaves the product \newline%
$p_{1}p_{2}\cdots p_{i}m_{1}m_{2}\cdots m_{\ell}q_{1}q_{2}\cdots q_{j}$
unchanged). Hence,
\[
m_{1}m_{2}\cdots m_{\ell}=m_{1}m_{2}\cdots\widehat{m_{a}}\cdots\widehat{m_{b}%
}\cdots m_{\ell},
\]
which shows that $m_{1}m_{2}\cdots m_{\ell}$ is not a reduced expression. This
contradicts the fact that $m_{1}m_{2}\cdots m_{\ell}$ is a reduced
expression.\ This contradiction shows that Case 1 is impossible.

\item An analogous argument shows that Case 2 is impossible (since $p_{1}%
p_{2}\cdots p_{i}$ is a reduced expression).

\item Similarly, Case 3 is impossible (since $q_{1}q_{2}\cdots q_{j}$ is a
reduced expression).

\item Assume that we are in Case 4. Thus, only one of the redundant factors
has the form $m_{a}$ for some $a\in\left\{  1,2,\ldots,\ell\right\}  $. Hence,
the other redundant factor cannot have the form $m_{b}$, and thus must have
the form $p_{b}$ or $q_{b}$ for some $b$. We may assume without loss of
generality that it has the form $p_{b}$ (since the other case is analogous).
Consider these $a$ and $b$. Thus,%
\begin{align*}
&  p_{1}p_{2}\cdots p_{i}m_{1}m_{2}\cdots m_{\ell}q_{1}q_{2}\cdots q_{j}\\
&  =p_{1}p_{2}\cdots\widehat{p_{b}}\cdots p_{i}m_{1}m_{2}\cdots\widehat{m_{a}%
}\cdots m_{\ell}q_{1}q_{2}\cdots q_{j}%
\end{align*}
(since the removal of the redundant factors leaves the product \newline%
$p_{1}p_{2}\cdots p_{i}m_{1}m_{2}\cdots m_{\ell}q_{1}q_{2}\cdots q_{j}$
unchanged). Hence,%
\begin{align*}
xmy  &  =p_{1}p_{2}\cdots p_{i}m_{1}m_{2}\cdots m_{\ell}q_{1}q_{2}\cdots
q_{j}\\
&  =\underbrace{p_{1}p_{2}\cdots\widehat{p_{b}}\cdots p_{i}}_{\substack{\in
W_{P}\\\text{(since all the }p_{k}\text{ belong to }P\text{)}}}m_{1}%
m_{2}\cdots\widehat{m_{a}}\cdots m_{\ell}\underbrace{q_{1}q_{2}\cdots q_{j}%
}_{\substack{\in W_{Q}\\\text{(since all the }q_{k}\text{ belong to }%
Q\text{)}}}\\
&  \in W_{P}\left(  m_{1}m_{2}\cdots\widehat{m_{a}}\cdots m_{\ell}\right)
W_{Q},
\end{align*}
so that%
\[
W_{P}\left(  m_{1}m_{2}\cdots\widehat{m_{a}}\cdots m_{\ell}\right)
W_{Q}=\underbrace{W_{P}x}_{\substack{=W_{P}\\\text{(since }x\in W_{P}\text{)}%
}}m\underbrace{yW_{Q}}_{\substack{=W_{Q}\\\text{(since }y\in W_{Q}\text{)}%
}}=W_{P}mW_{Q}=W_{P}wW_{Q}%
\]
(since $m\in W_{P}wW_{Q}$). Therefore, $m_{1}m_{2}\cdots\widehat{m_{a}}\cdots
m_{\ell}\in W_{P}wW_{Q}$. But this element $m_{1}m_{2}\cdots\widehat{m_{a}%
}\cdots m_{\ell}$ clearly has length $\leq\ell-1<\ell=\ell\left(  m\right)  $,
and thus cannot belong to $W_{P}wW_{Q}$ (since $m$ is an element of minimum
length in $W_{P}wW_{Q}$). This contradicts $m_{1}m_{2}\cdots\widehat{m_{a}%
}\cdots m_{\ell}\in W_{P}wW_{Q}$. This contradiction shows that Case 4 is impossible.
\end{itemize}

Thus we know that Cases 1, 2, 3 and 4 are all impossible. Hence, we must be in
Case 5. In other words, one of the redundant factors has the form $p_{a}$ for
some $a\in\left\{  1,2,\ldots,i\right\}  $, while the other has the form
$q_{b}$ for some $b\in\left\{  1,2,\ldots,j\right\}  $. Consider these $a$ and
$b$. Therefore,%
\begin{align*}
&  p_{1}p_{2}\cdots p_{i}m_{1}m_{2}\cdots m_{\ell}q_{1}q_{2}\cdots q_{j}\\
&  =p_{1}p_{2}\cdots\widehat{p_{a}}\cdots p_{i}m_{1}m_{2}\cdots m_{\ell}%
q_{1}q_{2}\cdots\widehat{q_{b}}\cdots q_{j}%
\end{align*}
(since the removal of the redundant factors leaves the product \newline%
$p_{1}p_{2}\cdots p_{i}m_{1}m_{2}\cdots m_{\ell}q_{1}q_{2}\cdots q_{j}$
unchanged). Hence,%
\begin{align*}
xmy  &  =p_{1}p_{2}\cdots p_{i}m_{1}m_{2}\cdots m_{\ell}q_{1}q_{2}\cdots
q_{j}\\
&  =p_{1}p_{2}\cdots\widehat{p_{a}}\cdots p_{i}\underbrace{m_{1}m_{2}\cdots
m_{\ell}}_{=m}q_{1}q_{2}\cdots\widehat{q_{b}}\cdots q_{j}\\
&  =p_{1}p_{2}\cdots\widehat{p_{a}}\cdots p_{i}\cdot m\cdot q_{1}q_{2}%
\cdots\widehat{q_{b}}\cdots q_{j}.
\end{align*}
If we set $x^{\prime}:=p_{1}p_{2}\cdots\widehat{p_{a}}\cdots p_{i}$ and
$y^{\prime}:=q_{1}q_{2}\cdots\widehat{q_{b}}\cdots q_{j}$, then this rewrites
as $xmy=x^{\prime}my^{\prime}$. Furthermore, $x^{\prime}\in W_{P}$ (since all
the $p_{k}$ belong to $P$) and $y^{\prime}\in W_{Q}$ (likewise) and
$x^{\prime}\leq x$ (by Proposition \ref{prop.bruhat-redex2}, since $x^{\prime
}=p_{1}p_{2}\cdots\widehat{p_{a}}\cdots p_{i}$ is a subproduct\ of the reduced
expression $p_{1}p_{2}\cdots p_{i}$ of $x$). Furthermore, $\ell\left(
x^{\prime}\right)  <\ell\left(  x\right)  $ (since $x^{\prime}=p_{1}%
p_{2}\cdots\widehat{p_{a}}\cdots p_{i}$ clearly has length $\leq
i-1<i=\ell\left(  x\right)  $) and therefore $x^{\prime}\neq x$, so that
$x^{\prime}<x$ (since $x^{\prime}\leq x$). Similarly, $y^{\prime}<y$. Hence,
we have found $x^{\prime}\in W_{P}$ and $y^{\prime}\in W_{Q}$ such that
$x^{\prime}<x$ and $y^{\prime}<y$ and $xmy=x^{\prime}my^{\prime}$. This proves
part \textbf{(d)}. \medskip

\textbf{(c)} Since $m \in W_{P}wW_{Q}$, it will suffice to show that all $u
\in W_{P} w W_{Q}$ satisfy $m \leq u$. So let us show this.

Let $u\in W_{P}wW_{Q}$. Part \textbf{(b)} shows that there exists a
factorization $u=xmy$ of $u$ with $x\in W_{P}$ and $y\in W_{Q}$. Let us pick
one such factorization such that $\ell\left(  x\right)  $ is minimal among all
such factorizations. The minimality of $\ell\left(  x\right)  $ shows that if
$x^{\prime}\in W_{P}$ and $y^{\prime}\in W_{Q}$ satisfy $u=x^{\prime
}my^{\prime}$, then%
\begin{equation}
\ell\left(  x^{\prime}\right)  \geq\ell\left(  x\right)  .
\label{pf.prop.double-coset.b.1}%
\end{equation}

If we had $\ell\left(  xmy\right)  \neq\ell\left(  x\right)  +\ell\left(
m\right)  +\ell\left(  y\right)  $, then part \textbf{(d)} would yield that
there exist $x^{\prime}\in W_{P}$ and $y^{\prime}\in W_{Q}$ such that
$x^{\prime}<x$ and $y^{\prime}<y$ and $xmy=x^{\prime}my^{\prime}$. This would
entail $u=xmy=x^{\prime}my^{\prime}$ and $\ell\left(  x^{\prime}\right)
<\ell\left(  x\right)  $ (by Proposition \ref{prop.bruhat.props1}
\textbf{(b)}, since $x^{\prime}<x$), which would contradict
(\ref{pf.prop.double-coset.b.1}). Thus, we cannot have $\ell\left(
xmy\right)  \neq\ell\left(  x\right)  +\ell\left(  m\right)  +\ell\left(
y\right)  $. Hence, we have $\ell\left(  xmy\right)  =\ell\left(  x\right)
+\ell\left(  m\right)  +\ell\left(  y\right)  $.

Let $p_{1}p_{2}\cdots p_{i}$ and $q_{1}q_{2}\cdots q_{j}$ and $m_{1}%
m_{2}\cdots m_{\ell}$ be reduced expressions for $x$ and $y$ and $m$,
respectively, where all the $p_{k}$ belong to $P$ (since $x\in W_{P}$), and
where all the $q_{k}$ belong to $Q$ (since $y\in W_{Q}$). Then, $\ell\left(
x\right)  =i$ and $\ell\left(  y\right)  =j$ and $\ell\left(  m\right)  =\ell
$, and furthermore%
\[
xmy=p_{1}p_{2}\cdots p_{i}m_{1}m_{2}\cdots m_{\ell}q_{1}q_{2}\cdots q_{j}%
\]
is a reduced expression for $xmy$ (since $\ell\left(  xmy\right)
=\underbrace{\ell\left(  x\right)  }_{=i}+\underbrace{\ell\left(  m\right)
}_{=\ell}+\underbrace{\ell\left(  y\right)  }_{=j}=i+\ell+j$). But removing
all the factors $p_{1},p_{2},\ldots,p_{i}$ and $q_{1},q_{2},\ldots,q_{j}$ from
this expression yields $m_{1}m_{2}\cdots m_{\ell}=m$. Thus, $m$ is a
subproduct\ of this reduced expression. Hence, by Proposition
\ref{prop.bruhat-redex2}, we obtain $m\leq xmy=u$. As we explained above, this
completes the proof of part \textbf{(c)}. \medskip

\textbf{(a)} We have already proved Proposition \ref{prop.double-coset}
\textbf{(c)} under the condition that $m$ is \textbf{any} element of minimum
length in $W_{P}wW_{Q}$. That is, we have shown that any element of minimum
length in $W_{P}wW_{Q}$ is a Bruhat minimum of $W_{P} w W_{Q}$. Hence, if the
set $W_{P} w W_{Q}$ had two distinct elements of minimum length, then it would
have two distinct Bruhat minima, which is absurd (since the Bruhat minimum of
a given subset $X$ of $W$ is unique when it exists). Thus, the set $W_{P} w
W_{Q}$ has at most one element of minimum length. Hence, this set has a unique
element of minimum length (since it is nonempty). This proves part
\textbf{(a)}.
\end{proof}

\section{Proof of Theorem \ref{thm.triple}}

Now, we prove our first main result:

\begin{proof}
[Proof of Theorem \ref{thm.triple}.]Proposition \ref{prop.double-coset}
\textbf{(a)} shows that the double coset $W_{I}uW_{J}$ has a unique element
$u^{\prime}$ of minimum length. Similarly, the double coset $W_{J}vW_{K}$ has
a unique element $v^{\prime}$ of minimum length. Consider these $u^{\prime}$
and $v^{\prime}$. Clearly, $W_{I}uW_{J}=W_{I}u^{\prime}W_{J}$ (since
$u^{\prime}\in W_{I}uW_{J}$) and $W_{J}vW_{K}=W_{J}v^{\prime}W_{K}$ (since
$v^{\prime}\in W_{J}vW_{K}$), so that
\[
\underbrace{W_{I}uW_{J}}_{=W_{I}u^{\prime}W_{J}}vW_{K}=W_{I}u^{\prime
}\underbrace{W_{J}vW_{K}}_{=W_{J}v^{\prime}W_{K}}=W_{I}u^{\prime}%
W_{J}v^{\prime}W_{K}.
\]
Hence,
\[
u^{\prime}v^{\prime}=\underbrace{1}_{\in W_{I}}u^{\prime}\underbrace{1}_{\in
W_{J}}v^{\prime}\underbrace{1}_{\in W_{K}}\in W_{I}u^{\prime}W_{J}v^{\prime
}W_{K}=W_{I}uW_{J}vW_{K}.
\]

Proposition \ref{prop.double-coset} \textbf{(a)} shows that the double coset
$W_{I}u^{\prime}v^{\prime}W_{K}$ has a unique element $m$ of minimum length.
Consider this $m$. Thus,%
\[
m\in W_{I}\underbrace{u^{\prime}v^{\prime}}_{=u^{\prime}1v^{\prime}}%
W_{K}=W_{I}u^{\prime}\underbrace{1}_{\in W_{J}}v^{\prime}W_{K}\subseteq
W_{I}u^{\prime}W_{J}v^{\prime}W_{K}=W_{I}uW_{J}vW_{K}.
\]

Proposition \ref{prop.double-coset} \textbf{(c)} shows that $m$ is the Bruhat
minimum of $W_{I}u^{\prime}v^{\prime}W_{K}$. Thus, each $r\in W_{I}u^{\prime
}v^{\prime}W_{K}$ satisfies $m\leq r$.

We shall now prove that every $p\in W_{I}uW_{J}vW_{K}$ satisfies $m\leq p$.
Once this is proved, it will follow that $m$ is the Bruhat minimum of
$W_{I}uW_{J}vW_{K}$ (since $m\in W_{I}uW_{J}vW_{K}$), and thus Theorem
\ref{thm.triple} will follow.

First, let us show the following:

\begin{statement}
\textit{Claim 1:} For every $x\in W_{I}$ and $y\in W_{J}$ and $z\in W_{K}$, we
have $m\leq xu^{\prime}yv^{\prime}z$.
\end{statement}

\begin{proof}
[Proof of Claim 1.]We proceed by strong induction on $\ell\left(  y\right)  $:

\textit{Base case:} It is easy to see that Claim 1 holds for $\ell\left(
y\right)  =0$. Indeed, if $\ell\left(  y\right)  =0$, then $y=1$ and thus
$xu^{\prime}\underbrace{y}_{=1}v^{\prime}z=\underbrace{x}_{\in W_{I}}%
u^{\prime}v^{\prime}\underbrace{z}_{\in W_{K}}\in W_{I}u^{\prime}v^{\prime
}W_{K}$, so that $m\leq xu^{\prime}yv^{\prime}z$ (since each $r\in
W_{I}u^{\prime}v^{\prime}W_{K}$ satisfies $m\leq r$).

\textit{Induction step:}\ Let $N$ be a positive integer. Assume (as the
induction hypothesis) that Claim 1 holds whenever $\ell\left(  y\right)  <N$.
We must now prove it for $\ell\left(  y\right)  =N$.

So let $x\in W_{I}$ and $y\in W_{J}$ and $z\in W_{K}$ be such that
$\ell\left(  y\right)  =N$. We must show that $m\leq xu^{\prime}yv^{\prime}z$.

Since $y\in W_{J}$ and $\ell\left(  y\right)  =N>0$, we can write $y$ as
$y=\overline{y}j$ for some $j\in J$ and some $\overline{y}\in W_{J}$
satisfying $\ell\left(  \overline{y}\right)  =\ell\left(  y\right)  -1$
(indeed, pick any reduced expression for $y$, and let $j$ be its last entry,
while $\overline{y}$ is the rest of this expression). Consider these $j$ and
$\overline{y}$. Of course, $j\in J\subseteq W_{J}$ and $\ell\left(  j\right)
=1$ (since $j\in J\subseteq S$) and $\ell\left(  \overline{y}\right)
+1=\ell\left(  y\right)  $ (since $\ell\left(  \overline{y}\right)
=\ell\left(  y\right)  -1$). Note also that $\ell\left(  \overline{y}\right)
=\ell\left(  y\right)  -1<\ell\left(  y\right)  =N$. Hence, our induction
hypothesis (applied to $\overline{y}$ instead of $y$) yields
\begin{equation}
m\leq xu^{\prime}\overline{y}v^{\prime}z. \label{pf.thm.triple.IH1}%
\end{equation}

We are in one of the following three cases:

\textit{Case 1:} We have $\ell\left(  xu^{\prime}y\right)  \neq\ell\left(
x\right)  +\ell\left(  u^{\prime}\right)  +\ell\left(  y\right)  $.

\textit{Case 2:} We have $\ell\left(  jv^{\prime}z\right)  \neq\ell\left(
j\right)  +\ell\left(  v^{\prime}\right)  +\ell\left(  z\right)  $.

\textit{Case 3:} None of the above holds.

Let us first consider Case 1. In this case, we have $\ell\left(  xu^{\prime
}y\right)  \neq\ell\left(  x\right)  +\ell\left(  u^{\prime}\right)
+\ell\left(  y\right)  $. Hence, Proposition \ref{prop.double-coset}
\textbf{(d)} (applied to $I$, $J$, $u$ and $u^{\prime}$ instead of $P$, $Q$,
$w$ and $m$) yields that there exist $x^{\prime}\in W_{I}$ and $y^{\prime}\in
W_{J}$ such that $x^{\prime}<x$ and $y^{\prime}<y$ and $xu^{\prime}%
y=x^{\prime}u^{\prime}y^{\prime}$. Consider these $x^{\prime}$ and $y^{\prime
}$. From $y^{\prime}<y$, we use Proposition \ref{prop.bruhat.props1}
\textbf{(b)} to obtain $\ell\left(  y^{\prime}\right)  <\ell\left(  y\right)
=N$. Thus, our induction hypothesis (applied to $x^{\prime}$ and $y^{\prime}$
instead of $x$ and $y$) yields $m\leq x^{\prime}u^{\prime}y^{\prime}v^{\prime
}z$. In view of $xu^{\prime}y=x^{\prime}u^{\prime}y^{\prime}$, we can rewrite
this as $m\leq xu^{\prime}yv^{\prime}z$. Thus, the induction step is complete
in Case 1.

Let us next consider Case 2. In this case, we have $\ell\left(  jv^{\prime
}z\right)  \neq\ell\left(  j\right)  +\ell\left(  v^{\prime}\right)
+\ell\left(  z\right)  $. Hence, Proposition \ref{prop.double-coset}
\textbf{(d)} (applied to $J$, $K$, $v$, $v^{\prime}$, $j$ and $z$ instead of
$P$, $Q$, $w$, $m$, $x$ and $y$) yields that there exist $x^{\prime}\in W_{J}$
and $y^{\prime}\in W_{K}$ such that $x^{\prime}<j$ and $y^{\prime}<z$ and
$jv^{\prime}z=x^{\prime}v^{\prime}y^{\prime}$ (since $j\in W_{J}$). Consider
these $x^{\prime}$ and $y^{\prime}$. From $x^{\prime}<j$, we use Proposition
\ref{prop.bruhat.props1} \textbf{(b)} to obtain $\ell\left(  x^{\prime
}\right)  <\ell\left(  j\right)  =1$ and thus $\ell\left(  x^{\prime}\right)
=0$. Hence, $x^{\prime}=1$. Thus, $jv^{\prime}z=x^{\prime}v^{\prime}y^{\prime
}$ rewrites as $jv^{\prime}z=v^{\prime}y^{\prime}$.

But our induction hypothesis (applied to $\overline{y}$ and $y^{\prime}$
instead of $y$ and $z$) yields $m\leq xu^{\prime}\overline{y}v^{\prime
}y^{\prime}$ (since $\overline{y}\in W_{J}$ and $y^{\prime}\in W_{K}$ and
$\ell\left(  \overline{y}\right)  <N$). In view of $\underbrace{y}%
_{=\overline{y}j}v^{\prime}z=\overline{y}\underbrace{jv^{\prime}z}%
_{=v^{\prime}y^{\prime}}=\overline{y}v^{\prime}y^{\prime}$, we can rewrite
this as $m\leq xu^{\prime}yv^{\prime}z$. Thus, the induction step is complete
in Case 2.

Finally, let us consider Case 3. In this case, none of the assumptions of the
previous two cases hold. Thus, we have%
\[
\ell\left(  xu^{\prime}y\right)  =\ell\left(  x\right)  +\ell\left(
u^{\prime}\right)  +\ell\left(  y\right)  \ \ \ \ \ \ \ \ \ \ \text{and}%
\ \ \ \ \ \ \ \ \ \ \ell\left(  jv^{\prime}z\right)  =\ell\left(  j\right)
+\ell\left(  v^{\prime}\right)  +\ell\left(  z\right)  .
\]
Furthermore, the triangle inequality for lengths (i.e., the fact that
$\ell\left(  ab\right)  \leq\ell\left(  a\right)  +\ell\left(  b\right)  $ for
all $a,b\in W$) yields%
\[
\ell\left(  xu^{\prime}\overline{y}\right)  \leq\ell\left(  x\right)
+\ell\left(  u^{\prime}\right)  +\ell\left(  \overline{y}\right)
\ \ \ \ \ \ \ \ \ \ \text{and}\ \ \ \ \ \ \ \ \ \ \ell\left(  v^{\prime
}z\right)  \leq\ell\left(  v^{\prime}\right)  +\ell\left(  z\right)  .
\]

Hence,
\[
\ell\left(  xu^{\prime}\overline{y}\right)  \leq\ell\left(  x\right)
+\ell\left(  u^{\prime}\right)  +\underbrace{\ell\left(  \overline{y}\right)
}_{<\ell\left(  y\right)  }<\ell\left(  x\right)  +\ell\left(  u^{\prime
}\right)  +\ell\left(  y\right)  =\ell\left(  xu^{\prime}\underbrace{y}%
_{=\overline{y}j}\right)  =\ell\left(  xu^{\prime}\overline{y}j\right)  .
\]
Since $j$ is a reflection (because $j\in J\subseteq S$), this entails
$xu^{\prime}\overline{y}<xu^{\prime}\overline{y}j$ (by Proposition
\ref{prop.bruhat.reflec} \textbf{(b)}, applied to $w=xu^{\prime}\overline{y}$
and $t=j$).

Furthermore,%
\[
\ell\left(  jv^{\prime}z\right)  =\underbrace{\ell\left(  j\right)  }%
_{=1>0}+\,\ell\left(  v^{\prime}\right)  +\ell\left(  z\right)  >\ell\left(
v^{\prime}\right)  +\ell\left(  z\right)  \geq\ell\left(  v^{\prime}z\right)
\]
(since $\ell\left(  v^{\prime}z\right)  \leq\ell\left(  v^{\prime}\right)
+\ell\left(  z\right)  $), so that $\ell\left(  v^{\prime}z\right)
<\ell\left(  jv^{\prime}z\right)  $. Since $j$ is a reflection, this entails
$v^{\prime}z<jv^{\prime}z$ (by Proposition \ref{prop.bruhat.reflec}
\textbf{(a)}, applied to $w=v^{\prime}z$ and $t=j$).

Now we know that $j$ is a reflection and we have $xu^{\prime}\overline
{y}<xu^{\prime}\overline{y}j$ and $v^{\prime}z<jv^{\prime}z$. Thus, Lemma
\ref{lem.bjobre22} (applied to $xu^{\prime}\overline{y}$, $v^{\prime}z$ and
$j$ instead of $x$, $y$ and $t$) yields $xu^{\prime}\overline{y}v^{\prime
}z<xu^{\prime}\overline{y}jv^{\prime}z$. Now, (\ref{pf.thm.triple.IH1})
becomes
\[
m\leq xu^{\prime}\overline{y}v^{\prime}z<xu^{\prime}\underbrace{\overline{y}%
j}_{=y}v^{\prime}z=xu^{\prime}yv^{\prime}z.
\]
Thus, the induction step is complete in Case 3.

We have now completed the induction step in all three cases. Thus, the
induction is complete, and Claim 1 is proved.
\end{proof}

Now, every $p\in W_{I}uW_{J}vW_{K}$ can be written in the form $p=xu^{\prime
}yv^{\prime}z$ for some $x\in W_{I}$ and $y\in W_{J}$ and $z\in W_{K}$
(because $p\in W_{I}uW_{J}vW_{K}=W_{I}u^{\prime}W_{J}v^{\prime}W_{K}$), and
thus satisfies $m\leq p$ (since Claim 1 yields $m\leq xu^{\prime}yv^{\prime}%
z$). Since we also know that $m\in W_{I}uW_{J}vW_{K}$, we thus conclude that
$m$ is the Bruhat minimum of the set $W_{I}uW_{J}vW_{K}$. Hence, Theorem
\ref{thm.triple} is proved.
\end{proof}

\section{Reminders on the longest element}

Next, we recall the definition of the so-called \emph{longest element} (aka
\emph{longest word}) in a finite Coxeter group:

\begin{proposition}
\label{prop.w0-exists} Let $W$ be a finite Coxeter group. Then, there is a
unique element $w\in W$ whose length $\ell\left(  w\right)  $ is maximum.
\end{proposition}

\begin{proof}
This is (contained in) \cite[\S 11.1]{Lusztig} or \cite[Proposition
2.3.1]{BjoBre} or \cite[Proposition 80]{Schremmer} or \cite[Lemma
4.6.1]{Davis} or \cite[Lemma 1.33]{Sam21}.
\end{proof}

The unique element $w$ in Proposition~\ref{prop.w0-exists} is called the
\emph{longest element} of $W$. It is commonly denoted $w_{0}$. Of all its
properties, we will need the following one:

\begin{proposition}
\label{prop.w0-l} Let $W$ be a finite Coxeter group. Let $w_{0}$ be its
longest element. Then, each $u\in W$ satisfies $\ell\left(  uw_{0}\right)
=\ell\left(  w_{0}\right)  -\ell\left(  u\right)  $.
\end{proposition}

\begin{proof}
This is \cite[\S 11.2 \textbf{(a)}]{Lusztig} or \cite[Proposition 80
(c)]{Schremmer} or \cite[Proposition 2.3.2 (ii)]{BjoBre}. (It also follows
easily from \cite[Lemma 4.6.1 (a)]{Davis}.)
\end{proof}

\begin{noncompile}
\begin{proposition}
\label{prop.w0} Let $W$ be a finite Coxeter group with Coxeter system $\left(
W,S\right)  $. Then:

\begin{enumerate}
\item[\textbf{(a)}] There is a unique element $w\in W$ whose length
$\ell\left(  w\right)  $ is maximum.
\end{enumerate}

Let us denote this element by $w_{0}$. It is called the \emph{longest element}
of $W$. It has the following properties:

\begin{enumerate}
\item[\textbf{(b)}] We have $w_{0}^{2}=1$, that is, $w_{0}^{-1}=w_{0}$.

\item[\textbf{(c)}] Each $u\in W$ satisfies $u\leq w_{0}$.

\item[\textbf{(d)}] Each $u\in W$ satisfies $\ell\left(  uw_{0}\right)
=\ell\left(  w_{0}u\right)  =\ell\left(  w_{0}\right)  -\ell\left(  u\right)
$.

\item[\textbf{(e)}] Each $u\in W$ satisfies $\ell\left(  w_{0}uw_{0}%
^{-1}\right)  =\ell\left(  u\right)  $.

\item[\textbf{(f)}] We have $w_{0}Sw_{0}^{-1}=S$.

\item[\textbf{(g)}] For any $u,v\in W$, we have the logical equivalences%
\[
\left(  u\leq v\right)  \ \Longleftrightarrow\ \left(  w_{0}u\geq
w_{0}v\right)  \ \Longleftrightarrow\ \left(  uw_{0}\geq vw_{0}\right)  .
\]

\end{enumerate}
\end{proposition}

\begin{proof}
\textbf{(a)} This is (contained in) \cite[\S 11.1]{Lusztig} or
\cite[Proposition 2.3.1]{BjoBre} or \cite[Proposition 80]{Schremmer} or
\cite[Lemma 4.6.1]{Davis}.

\textbf{(g)} This is \cite[Lemma 11.3 \textbf{(a)}]{Lusztig} or
\cite[Proposition 2.3.4]{BjoBre}.

\textbf{(b)} This is \cite[Proposition 2.3.2 (i)]{BjoBre} or \cite[Proposition
80 \textbf{(a)}]{Schremmer} or \cite[Lemma 4.6.1 (iii)]{Davis}. (It is also an
immediate consequence of part \textbf{(a)}, since $w_{0}^{-1}$ has the same
length as $w_{0}$.)

\textbf{(c)} This is \cite[Proposition 2.3.1 (i)]{BjoBre} or \cite[Proposition
80 (e)]{Schremmer}. It also follows from parts \textbf{(b)} and \textbf{(g)}
as follows: Part \textbf{(b)} yields $w_{0}^{2}=1$. Now let $u\in W$. Then,
$1\leq w_{0}u$ (since $1\leq v$ for any $v\in W$). Hence, part \textbf{(g)}
(applied to $1$ and $w_{0}u$ instead of $u$ and $v$) yields that $w_{0}1\geq
w_{0}w_{0}u$ and $1w_{0}\geq w_{0}uw_{0}$. In particular, $w_{0}%
1\geq\underbrace{w_{0}w_{0}}_{=w_{0}^{2}=1}u=u$, hence $u\leq w_{0}1=w_{0}$.
This proves part \textbf{(c)}.

\textbf{(d)} This is \cite[\S 11.2 \textbf{(a)} and \textbf{(b)}]{Lusztig} or
\cite[Proposition 80 (c)]{Schremmer} or \cite[Proposition 2.3.2 (ii) and
Corollary 2.3.3 (i)]{BjoBre}.

\textbf{(e)} Part \textbf{(b)} yields $w_{0}^{2}=1$ and $w_{0}^{-1}=w_{0}$.

Let $u\in W$. Then, the first equality of part \textbf{(d)} (applied to
$w_{0}u$ instead of $u$) yields $\ell\left(  w_{0}uw_{0}\right)  =\ell\left(
w_{0}w_{0}u\right)  $. In view of $\underbrace{w_{0}w_{0}}_{=w_{0}^{2}=1}u=u$,
we can rewrite this as $\ell\left(  w_{0}uw_{0}\right)  =\ell\left(  u\right)
$. In other words, $\ell\left(  w_{0}uw_{0}^{-1}\right)  =\ell\left(
u\right)  $ (since $w_{0}^{-1}=w_{0}$). This proves part \textbf{(e)}.

\textbf{(f)} Let $s\in S$. Then, $\ell\left(  s\right)  =1$, so that part
\textbf{(e)} (applied to $u=s$) yields $\ell\left(  w_{0}sw_{0}^{-1}\right)
=\ell\left(  s\right)  =1$. This shows that $w_{0}sw_{0}^{-1}\in S$. Since we
have proved this for each $s\in S$, we thus have shown that $w_{0}Sw_{0}%
^{-1}\subseteq S$.

Since $w_{0}^{-1}=w_{0}$ (by part \textbf{(b)}), we can rewrite this as
$w_{0}^{-1}Sw_{0}\subseteq S$. In other words, $S\subseteq w_{0}Sw_{0}^{-1}$.
Combining this with $w_{0}Sw_{0}^{-1}\subseteq S$, we obtain $w_{0}Sw_{0}%
^{-1}=S$. This proves part \textbf{(f)}.
\end{proof}
\end{noncompile}

\section{Classical coset lemmas}

We shall next prove a sequence of lemmas -- most of them likely well-known --
that will be used in the proofs of our other two results. The first is a
classic (see, e.g., \cite[Lemma 9.7 (b)]{Lusztig}):

\begin{lemma}
\label{len-add} Let $I$ be a subset of $S$. Let $w\in W$. Let $a$ and $b$ be
elements of minimum length in the cosets $wW_{I}$ and $W_{I}w$, respectively. Then:

\begin{enumerate}
\item[\textbf{(a)}] For each $x\in W_{I}$, we have $\ell(xb)=\ell(x)+\ell(b)$.

\item[\textbf{(b)}] For each $y\in W_{I}$, we have $\ell(ay)=\ell(a)+\ell(y)$.
\end{enumerate}
\end{lemma}

\begin{proof}
\textbf{(a)} Let $x\in W_{I}$. We must prove that $\ell(xb)=\ell(x)+\ell(b)$.
Assume the contrary. Thus, $\ell(xb)\neq\ell(x)+\ell(b)$. So
\[
\ell(xb1)=\ell\left(  xb\right)  \neq\ell\left(  x\right)  +\ell\left(
b\right)  =\ell(x)+\ell(b)+\ell\left(  1\right)  \ \ \ \ \ \ \ \ \ \ \left(
\text{since }\ell\left(  1\right)  =0\right)  .
\]

Since $b\in W_{I}w$, we have $W_{I}b=W_{I}w=W_{I}wW_{\varnothing}$ (since
$W_{\varnothing}=\left\{  1\right\}  $). Proposition \ref{prop.double-coset}
\textbf{(a)} (applied to $P=I$ and $Q=\varnothing$) says that the double coset
$W_{I}wW_{\varnothing}$ has a unique element of minimum length. This unique
element must be $b$ (since $b$ has minimum length in $W_{I}w=W_{I}%
wW_{\varnothing}$). Since $\ell(xb1)\neq\ell(x)+\ell(b)+\ell\left(  1\right)
$, we can thus apply Proposition \ref{prop.double-coset} \textbf{(d)} to $P=I$
and $Q=\varnothing$ and $m=b$ and $y=1$. This yields that there exist
$x^{\prime}\in W_{I}$ and $y^{\prime}\in W_{\varnothing}$ such that
$x^{\prime}<x$ and $y^{\prime}<1$ and $xb1=x^{\prime}by^{\prime}$. But this is
clearly absurd, since $y^{\prime}<1$ is impossible (no element of $W$ is
smaller than $1$). This contradiction shows that our assumption was false. So
Lemma \ref{len-add} \textbf{(a)} is proved. \medskip

\textbf{(b)} This is similar to part \textbf{(a)}, but with the left and right
sides swapped (i.e., use $wW_{I}=W_{\varnothing}wW_{I}$, and rewrite $ay$ as
$1ay$).
\end{proof}

The following lemma says that any element in a finite parabolic subgroup can
be connected to the longest element of that subgroup by a sequence of
multiplications by simple reflections, each of which increases the length by
$1$.

\begin{lemma}
\label{max-ext}Let $I\subseteq S$ be such that $W_{I}$ is finite. Let $y\in
W_{I}$, and let $w_{I}$ be the longest element of $W_{I}$. Then, for some
integer $r\geq0$, there exist $y_{0},y_{1},\dots,y_{r}\in W_{I}$ and
$i_{0},i_{1},\dots,i_{r-1}\in I$ with $y_{0}=y$ and $y_{r}=w_{I}$, so that
\[
y_{j+1}=y_{j}i_{j}\quad\text{and}\quad\ell(y_{j+1})=\ell(y_{j}%
)+1\ \ \ \ \ \ \ \ \ \ \text{for every }j\in\left\{  0,1,\dots,r-1\right\}  .
\]

\end{lemma}

\begin{proof}
We have $y^{-1}w_{I}\in W_{I}$ (since $y\in W_{I}$ and $w_{I}\in W_{I}$); thus
we can choose a reduced expression
\[
y^{-1}w_{I}=i_{0}i_{1}\cdots i_{r-1}%
\]
for $y^{-1}w_{I}$ with $i_{0},i_{1},\dots,i_{r-1}\in I$. Then, $\ell
(y^{-1}w_{I})=r$ and
\[
w_{I}=yi_{0}i_{1}\cdots i_{r-1}.
\]

By Proposition \ref{prop.w0-l} (applied to $W_{I}$, $w_{I}$ and $y^{-1}$ in
place of $W$, $w_{0}$ and $u$, respectively), we have
\[
\ell(y^{-1}w_{I})=\ell(w_{I})-\ell(y^{-1})=\ell(w_{I})-\ell
(y)\ \ \ \ \ \ \ \ \ \ \left(  \text{since }\ell\left(  y^{-1}\right)
=\ell\left(  y\right)  \right)
\]
and therefore
\[
\ell(w_{I})-\ell(y)=\ell(y^{-1}w_{I})=r.
\]


Next, for all $j\in\left\{  0,1,\ldots,r\right\}  $, set
\[
y_{j}=yi_{0}i_{1}\cdots i_{j-1}\in W_{I}.
\]
In particular, $y_{0}=y$, and $y_{r}=yi_{0}i_{1}\cdots i_{r-1}=w_{I}$. Note
that for all $j\in\left\{  0,1,\dots,r-1\right\}  $, we have $y_{j+1}%
=y_{j}i_{j}$ and therefore $\ell(y_{j+1})=\ell\left(  y_{j}i_{j}\right)
=\ell(y_{j})\pm1$ (since $i_{j}\in I\subseteq S$), thus
\begin{equation}
\ell(y_{j+1})-\ell(y_{j})=\pm1\leq1. \label{yj+1yjineq}%
\end{equation}
But the sum of these inequalities \eqref{yj+1yjineq} over all $j\in\left\{
0,1,\ldots,r-1\right\}  $ is $\ell(y_{r})-\ell(y_{0})\leq r$, which is an
equality because $\ell(y_{r})-\ell(y_{0})=\ell(w_{I})-\ell(y)=r$. Thus, each
of the inequalities \eqref{yj+1yjineq} must have already been an equality.
That is, $\ell(y_{j+1})=\ell(y_{j})+1$ holds for every $j$. So the elements
$i_{j}$ and $y_{j}$ satisfy all required properties.
\end{proof}

\begin{lemma}
\label{lem-sandwich}Let $I\subseteq S$ be such that $W_{I}$ is finite. Let
$a,b\in W$. Let $w_{I}$ be the longest element of $W_{I}$, and let $a_{0}$ and
$b_{0}$ be elements of minimum length in the cosets $aW_{I}$ and $W_{I}b$,
respectively. Then, we have $aW_{I}b=a_{0}W_{I}b_{0}$, and $a_{0}w_{I}b_{0}$
is the Bruhat maximum of this set.
\end{lemma}

\begin{proof}
First, note that $aW_{I}=a_{0}W_{I}$ (since $a_{0}\in aW_{I}$) and
$W_{I}b=W_{I}b_{0}$ (since $b_{0}\in W_{I}b$). This implies that
\[
aW_{I}b=a_{0}W_{I}b=a_{0}W_{I}b_{0}.
\]

It remains to show that $a_{0}w_{I}b_{0}$ is the Bruhat maximum of this set
$a_{0}W_{I}b_{0}$. Since $a_{0}w_{I}b_{0}\in a_{0}W_{I}b_{0}$, this will be
done once we have proved that all $p\in a_{0}W_{I}b_{0}$ satisfy $p\leq
a_{0}w_{I}b_{0}$.

So let $p\in a_{0}W_{I}b_{0}$. Thus, $p=a_{0}yb_{0}$ for some $y\in W_{I}$.
Consider this $y$.

Choose $r\geq0$ and $y_{0},y_{1},\dots,y_{r}\in W_{I}$ and $i_{0},i_{1}%
,\dots,i_{r-1}\in I$ as in Lemma \ref{max-ext}. Then, $y_{0}=y$ and
$y_{r}=w_{I}$, and%
\[
y_{j+1}=y_{j}i_{j}\quad\text{and}\quad\ell(y_{j+1})=\ell(y_{j}%
)+1\ \ \ \ \ \ \ \ \ \ \text{for every }j\in\left\{  0,1,\dots,r-1\right\}  .
\]

Let $j\in\left\{  0,1,\ldots,r-1\right\}  $. Since $a_{0}$ has minimum length
in $aW_{I}$, we can apply Lemma \ref{len-add} \textbf{(b)} to $a$, $a_{0}$ and
$y_{j}$ in place of $w$, $a$ and $y$; thus we find%
\[
\ell(a_{0}y_{j})=\ell(a_{0})+\ell(y_{j}).
\]
Applying the same lemma to $y_{j+1}$ instead of $y_{j}$, we obtain%
\[
\ell(a_{0}y_{j+1})=\ell(a_{0})+\ell(y_{j+1}).
\]
In view of $y_{j+1}=y_{j}i_{j}$ and $\ell(y_{j+1})=\ell(y_{j})+1$, this
rewrites as
\[
\ell(a_{0}y_{j}i_{j})=\ell(a_{0})+\ell(y_{j})+1>\ell\left(  a_{0}\right)
+\ell\left(  y_{j}\right)  =\ell(a_{0}y_{j}).
\]
That is, $\ell\left(  a_{0}y_{j}\right)  <\ell(a_{0}y_{j}i_{j})$. Since
$i_{j}\in I\subseteq S$ is a reflection, this entails
\begin{equation}
a_{0}y_{j}<a_{0}y_{j}i_{j} \label{pf.lem-sandwich.4}%
\end{equation}
(by Proposition \ref{prop.bruhat.reflec} \textbf{(b)}, applied to $a_{0}y_{j}$
in place of $w$ and $i_{j}$ in place of $t$).

Furthermore, $i_{j}\in I\subseteq W_{I}$, so that $i_{j}b_{0}\in W_{I}%
b_{0}=W_{I}b$ (since $b_{0}\in W_{I}b$). Therefore, $\ell\left(  i_{j}%
b_{0}\right)  \geq\ell\left(  b_{0}\right)  $ (since $b_{0}$ has minimum
length in $W_{I}b$). Moreover, since $i_{j}$ is a simple reflection, we have
$\ell\left(  i_{j}b_{0}\right)  =\ell\left(  b_{0}\right)  \pm1$. In view of
the inequality $\ell\left(  i_{j}b_{0}\right)  \geq\ell\left(  b_{0}\right)
$, this $\pm$ sign must be a $+$ sign, so we obtain $\ell(i_{j}b_{0}%
)=\ell(b_{0})+1>\ell\left(  b_{0}\right)  $. That is, $\ell\left(
b_{0}\right)  <\ell\left(  i_{j}b_{0}\right)  $. Since $i_{j}\in I\subseteq S$
is a reflection, this entails
\begin{equation}
b_{0}<i_{j}b_{0} \label{pf.lem-sandwich.5}%
\end{equation}
(by Proposition \ref{prop.bruhat.reflec} \textbf{(a)}, applied to $b_{0}$ in
place of $w$ and $i_{j}$ in place of $t$).

Now, we can use the inequalities (\ref{pf.lem-sandwich.4}) and
(\ref{pf.lem-sandwich.5}) to apply Lemma \ref{lem.bjobre22} with $a_{0}y_{j}$
in place of $x$, $i_{j}$ in place of $t$, and $b_{0}$ in place of $y$ (since
$i_{j}\in I$ is a reflection), and we obtain
\[
a_{0}y_{j}b_{0}<a_{0}y_{j}i_{j}b_{0}=a_{0}y_{j+1}b_{0}%
\ \ \ \ \ \ \ \ \ \ \left(  \text{since }y_{j}i_{j}=y_{j+1}\right)  .
\]
Since this is true for all $j\in\left\{  0,1,\ldots,r-1\right\}  $, we obtain
\[
a_{0}y_{0}b_{0}<a_{0}y_{1}b_{0}<\cdots<a_{0}y_{r}b_{0},
\]
hence $a_{0}y_{0}b_{0}\leq a_{0}y_{r}b_{0}$ (this includes the case $r=0$, in
which case this is an equality). Since $y_{0}=y$ and $y_{r}=w_{I}$, this
rewrites as $a_{0}yb_{0}\leq a_{0}w_{I}b_{0}$. So $p=a_{0}yb_{0}\leq
a_{0}w_{I}b_{0}$. As explained, this completes the proof. \qedhere

\end{proof}


If $w \in W$ and $s \in S$, then we necessarily have $ws < w$ or $ws > w$
(indeed, $\ell(ws)$ is either $\ell(w) - 1$ or $\ell(w) + 1$; then
Proposition~\ref{prop.bruhat.reflec} \textbf{(b)} yields that we have $ws < w$
in the former case and $ws > w$ in the latter).


\begin{lemma}
\label{max.reflection}

Consider $v,w\in W$ so that $v\leq w$, and let $s\in S$. Then:

\begin{enumerate}
\item[\textbf{(a)}] If $sw<w$, then $sv\leq w$.

\item[\textbf{(b)}] If $ws<w$, then $vs\leq w$.
\end{enumerate}
\end{lemma}

\begin{proof}
\textbf{(a)} This is well-known. Indeed, assume that $sw<w$. Then,
\cite[Corollary 2.5 (a)]{Lusztig} (applied to $y=v$ and $z=w$) says that
$v\leq w$ is equivalent to $sv\leq w$. Thus, in particular, the former implies
the latter. \medskip






\textbf{(b)} Assume that $ws<w$. Then, Proposition \ref{prop.bruhat.props1}
\textbf{(d)} (applied to $ws$ in place of $u$, and $w$ in place of $v$) yields%
\[
(ws)^{-1}<w^{-1}.
\]
But $s\in S$ is a simple reflection, so it satisfies $s^{-1}=s$. Now, $\left(
ws\right)  ^{-1}=s^{-1}w^{-1}=sw^{-1}$ (since $s^{-1}=s$), so that%
\[
sw^{-1}=(ws)^{-1}<w^{-1}.
\]

Next, recall that $v\leq w$; thus, by Proposition \ref{prop.bruhat.props1}
\textbf{(c)}, we obtain $v^{-1}\leq w^{-1}$. Because of this and
$sw^{-1}<w^{-1}$, we can apply part \textbf{(a)} to $v^{-1}$ and $w^{-1}$
instead of $v$ and $w$; this yields%
\[
sv^{-1}\leq w^{-1}.
\]

We can again apply Proposition \ref{prop.bruhat.props1} \textbf{(c)}, with
$sv^{-1}$ in place of $u$, and $w^{-1}$ in place of $v$, to find that
$(sv^{-1})^{-1}\leq(w^{-1})^{-1}$. Simplifying this using $s^{-1}=s$, we get
that $vs\leq w$.
\end{proof}

\begin{lemma}
\label{inequality.thickening}

Choose $I\subseteq S$ so that $W_{I}$ is finite, and let $v,w\in W$ be such
that $v\leq w$. Note that the cosets $W_{I}w=1W_{I}w$ and $wW_{I}=wW_{I}1$
have Bruhat maxima by Lemma \ref{lem-sandwich}.

\begin{enumerate}
\item[\textbf{(a)}] Let $M$ be the Bruhat maximum of $W_{I}w$. Then, each
$x\in W_{I}$ satisfies $xv\leq M$.

\item[\textbf{(b)}] Let $N$ be the Bruhat maximum of $wW_{I}$. Then, each
$x\in W_{I}$ satisfies $vx\leq N$.
\end{enumerate}
\end{lemma}

\begin{proof}
\textbf{(a)} We induct on the length of $x$. First, if $x=1$, then $xv=v\leq
w\leq M$ (since $w\in W_{I}w$, and $M$ is the Bruhat maximum of $W_{I}w$).

Next, assume that $x\neq1$. Choose $t\in I$ and $x^{\prime}\in W_{I}$ so that
$x=tx^{\prime}$ and $\ell(x^{\prime})=\ell(x)-1$ (namely, let $t$ be the first
letter of a reduced expression for $x$, and let $x^{\prime}$ be the rest of
this expression). By our induction hypothesis, $x^{\prime}v\leq M$. Note that
because $M\in W_{I}w$ and $t\in I\subseteq W_{I}$, we have $tM\in W_{I}%
W_{I}w=W_{I}w$, and thus $tM\leq M$ because $M$ is the Bruhat maximum of
$W_{I}w$. Hence, $tM<M$ (since $t\neq1$ entails $tM\neq M$). Thus, we can
apply Lemma \ref{max.reflection} \textbf{(a)} (with $x^{\prime}v$ in place of
$v$, $M$ in place of $w$, and $t$ in place of $s$) and find that $tx^{\prime
}v\leq M$. Recall that $tx^{\prime}=x$, and so this rewrites as $xv\leq M$.
This completes the induction, thus proving \textbf{(a)}.

\textbf{(b)} This is analogous to part \textbf{(a)}; just read all products
right-to-left (in particular, $W_{I}w$ becomes $wW_{I}$) and apply Lemma
\ref{max.reflection} \textbf{(b)} instead of Lemma \ref{max.reflection}
\textbf{(a)}. (Of course, $t$ should now be the last, not the first, letter of
a reduced expression for $x$.)
\end{proof}

\begin{lemma}
\label{double.thickening}Choose $P,Q\subseteq S$ so that $W_{P}$ and $W_{Q}$
are finite, and consider $r,s\in W$ so that $r\leq s$. Then:

\begin{enumerate}
\item[\textbf{(a)}] The double coset $W_{P}sW_{Q}$ has a Bruhat maximum.

\item[\textbf{(b)}] Let $M$ be this maximum. Then, for every $x\in W_{P}$ and
$z\in W_{Q}$, we have $xrz\leq M$.
\end{enumerate}
\end{lemma}

\begin{proof}
The coset $sW_{Q}$ can be rewritten as $sW_{Q}1$, and thus has a Bruhat
maximum by Lemma \ref{lem-sandwich} (applied to $I=Q$ and $a=s$ and $b=1$); we
denote this maximum by $N$. Similarly, the coset $W_{P}N$ (which can be
written as $1W_{P}N$) has a Bruhat maximum, which we will call $M$. Note that
$M\in W_{P}N\subseteq W_{P}sW_{Q}$ (since $N\in sW_{Q}$). We shall show that
$M$ is the Bruhat maximum of $W_{P}sW_{Q}$.

Let $x\in W_{P}$ and $z\in W_{Q}$. Because $r\leq s$, we can apply Lemma
\ref{inequality.thickening} \textbf{(b)} with $I=Q$, $v=r$, $w=s$, and $x=z$,
and we find that $rz\leq N$. Consequently, we can apply Lemma
\ref{inequality.thickening} \textbf{(a)} with $I=P$, $v=rz$, and $w=N$, and we
get that
\begin{equation}
xrz\leq M. \label{star2}%
\end{equation}
Note that the definitions of $N$ and $M$ did not depend on $r$. Thus, we can
apply (\ref{star2}) to $r=s$ (since $s\leq s$), and obtain $xsz\leq M$.

So we have shown that $xsz\leq M$ for all $x\in W_{P}$ and $z\in W_{Q}$. In
other words, each element of $W_{P}sW_{Q}$ is $\leq M$. Since $M\in
W_{P}sW_{Q}$, this shows that $M$ is the Bruhat maximum of $W_{P}sW_{Q}$. This
proves part \textbf{(a)}. Furthermore, part \textbf{(b)} now follows from
(\ref{star2}).
\end{proof}

\section{Proof of Theorem \ref{thm.double-interval}}

Next we take aim at Theorem \ref{thm.double-interval}. First we establish a
useful lemma that allows us to \textquotedblleft disassemble\textquotedblright%
\ a Bruhat inequality:

\begin{lemma}
\label{lem.leqexist}Let $p,m,q,w\in W$ be such that $w\leq pmq$. Then, we can
find $p^{\prime},m^{\prime},q^{\prime}\in W$ satisfying $p^{\prime}\leq p$ and
$m^{\prime}\leq m$ and $q^{\prime}\leq q$, and $w=p^{\prime}m^{\prime
}q^{\prime}$.
\end{lemma}

\begin{proof}
Let $p_{1}\cdots p_{i}$ and $m_{1}\cdots m_{j}$ and $q_{1}\cdots q_{k}$ be
reduced expressions for $p$ and $m$ and $q$, respectively. Then,
\begin{equation}
pmq=p_{1}\cdots p_{i}m_{1}\cdots m_{j}q_{1}\cdots q_{k}
\label{pf.lem.leqexist.0}%
\end{equation}
is not necessarily a reduced expression. If it is not, then, by the deletion
property (Theorem~\ref{thm.del}), we can delete two factors from it without
changing the product. By applying this repeatedly, we eventually arrive at a
reduced expression of $pmq$ that has the form
\begin{equation}
pmq=p_{u_{1}}\cdots p_{u_{x}}m_{v_{1}}\cdots m_{v_{y}}q_{w_{1}}\cdots
q_{w_{z}} \label{pf.lem.leqexist.1}%
\end{equation}
for some $1\leq u_{1}<\cdots<u_{x}\leq i$ and $1\leq v_{1}<\cdots<v_{y}\leq j$
and $1\leq w_{1}<\cdots<w_{z}\leq k$. Consider this expression, and set
$p^{\ast}=p_{u_{1}}\cdots p_{u_{x}}$ and $m^{\ast}=m_{v_{1}}\cdots m_{v_{y}}$
and $q^{\ast}=q_{w_{1}}\cdots q_{w_{z}}$. Then, $pmq=p^{\ast}m^{\ast}q^{\ast}$
(by (\ref{pf.lem.leqexist.1})) and $p^{\ast}\leq p$ (since $p^{\ast}=p_{u_{1}%
}\cdots p_{u_{x}}$ is a subproduct of the reduced expression $p_{1}\cdots
p_{i}$ of $p$) and likewise $m^{\ast}\leq m$ and $q^{\ast}\leq q$.

Since $w\leq pmq$, we can write $w$ as a subproduct of the reduced expression
\[
p_{u_{1}}\cdots p_{u_{x}}m_{v_{1}}\cdots m_{v_{y}}q_{w_{1}}\cdots q_{w_{z}}%
\]
of $pmq$ (by Proposition \ref{prop.bruhat-redex2}). That is, we can write $w$
as
\begin{equation}
w=p_{u_{1}^{\prime}}\cdots p_{u_{a}^{\prime}}m_{v_{1}^{\prime}}\cdots
m_{v_{b}^{\prime}}q_{w_{1}^{\prime}}\cdots q_{w_{c}^{\prime}},
\label{pf.lem.leqexist.5}%
\end{equation}
where $p_{u_{1}^{\prime}}\cdots p_{u_{a}^{\prime}}$ is a subproduct of
$p_{u_{1}}\cdots p_{u_{x}}$, where $m_{v_{1}^{\prime}}\cdots m_{v_{b}^{\prime
}}$ is a subproduct of $m_{v_{1}}\cdots m_{v_{y}}$, and where $q_{w_{1}%
^{\prime}}\cdots q_{w_{c}^{\prime}}$ is a subproduct of $q_{w_{1}}\cdots
q_{w_{z}}$.

Set $p^{\prime}:=p_{u_{1}^{\prime}}\cdots p_{u_{a}^{\prime}}$ and $m^{\prime
}:=m_{v_{1}^{\prime}}\cdots m_{v_{b}^{\prime}}$ and $q^{\prime}:=q_{w_{1}%
^{\prime}}\cdots q_{w_{c}^{\prime}}$. Note that $p^{\prime}=p_{u_{1}^{\prime}%
}\cdots p_{u_{a}^{\prime}}$ is a subproduct of $p_{u_{1}}\cdots p_{u_{x}}$,
which in turn is a subproduct of $p_{1}\cdots p_{i}$, which is a reduced
expression of $p$. Thus, $p^{\prime}$ is a subproduct of a reduced expression
of $p$. Hence, by Proposition \ref{prop.bruhat-redex2}, we have $p^{\prime
}\leq p$; similarly, $m^{\prime}\leq m$ and $q^{\prime}\leq q$. Moreover,
\eqref{pf.lem.leqexist.5} rewrites as $w=p^{\prime}m^{\prime}q^{\prime}$ (by
the definitions of $p^{\prime},m^{\prime},q^{\prime}$). So we have found the
desired $p^{\prime},m^{\prime},q^{\prime}$. Lemma \ref{lem.leqexist} is thus proved.
\end{proof}

Now we are ready to prove Theorem \ref{thm.double-interval}:

\begin{proof}
[Proof of Theorem \ref{thm.double-interval}.]Proposition
\ref{prop.double-coset} \textbf{(c)} yields that the double coset $W_{I}%
wW_{J}$ has a Bruhat minimum. Let $m$ be this minimum.

By Lemma~\ref{double.thickening} \textbf{(a)} (applied to $P=I$, $Q=J$, $s=w$
and $r=w$), we furthermore know that the double coset $W_{I}wW_{J}$ has a
Bruhat maximum. Let $M$ be this maximum.




We shall now show that $W_{I}wW_{J}$ is the interval $\left[  m,M\right]  $ in
the Bruhat order.

\begin{noncompile}
(All inequalities in this proof are understood with respect to the Bruhat order.)
\end{noncompile}

Note that $W_{I}MW_{J}=W_{I}wW_{J}$ (since $M\in W_{I}wW_{J}$) and
$W_{I}mW_{J}=W_{I}wW_{J}$ (likewise).

Each $x\in W_{I}wW_{J}$ satisfies $m\leq x$ (since $m$ is the Bruhat minimum
of $W_{I}wW_{J}$) and $x\leq M$ (since $M$ is the Bruhat maximum of
$W_{I}wW_{J}$), thus $x\in\left[  m,M\right]  $. In other words,
\begin{equation}
W_{I}wW_{J}\subseteq\left[  m,M \right]  . \label{pf.thm.double-interval.easy}%
\end{equation}
We shall now prove the reverse inclusion.

Let $u\in\left[  m,M\right]  $. Then, $u\leq M$. But $M\in W_{I}wW_{J}%
=W_{I}mW_{J}$, so we can write $M$ as $pmq$ with $p\in W_{I}$ and $q\in W_{J}%
$. Consider these $p$ and $q$. Then, $u\leq M=pmq$. Hence, by Lemma
\ref{lem.leqexist} (applied to $u$ instead of $w$), we can find $p^{\prime
},m^{\prime},q^{\prime}\in W$ satisfying $p^{\prime}\leq p$ and $m^{\prime
}\leq m$ and $q^{\prime}\leq q$ and $u=p^{\prime}m^{\prime}q^{\prime}$.
Consider these $p^{\prime},m^{\prime},q^{\prime}$.

Since $p\in W_{I}$, there is a reduced expression for $p$ using only
reflections in $I$. Since $p^{\prime}\leq p$, we can write $p^{\prime}$ as a
subproduct of this expression. Thus, $p^{\prime}$ is a product of reflections
in $I$, whence $p^{\prime}\in W_{I}$. Similarly, from $q^{\prime}\leq q$ and
$q\in W_{J}$, we obtain $q^{\prime}\in W_{J}$.

Next, from $u\in\left[  m,M\right]  $, we obtain $m\leq u=p^{\prime}m^{\prime
}q^{\prime}$. Hence, Lemma \ref{lem.leqexist} (applied to $p^{\prime}$,
$m^{\prime}$, $q^{\prime}$ and $m$ instead of $p$, $m$, $q$ and $w$) shows
that we can find $p^{\prime\prime},m^{\prime\prime},q^{\prime\prime}\in W$
satisfying $p^{\prime\prime}\leq p^{\prime}$ and $m^{\prime\prime}\leq
m^{\prime}$ and $q^{\prime\prime}\leq q^{\prime}$ and $m=p^{\prime\prime
}m^{\prime\prime}q^{\prime\prime}$. Consider these $p^{\prime\prime}%
,m^{\prime\prime},q^{\prime\prime}$.

Previously we obtained $p^{\prime}\in W_{I}$ from $p\in W_{I}$ and $p^{\prime
}\leq p$. Likewise, we can now obtain $p^{\prime\prime}\in W_{I}$ from
$p^{\prime}\in W_{I}$ and $p^{\prime\prime}\leq p^{\prime}$. Similarly,
$q^{\prime\prime}\in W_{J}$. Hence, $m=\underbrace{p^{\prime\prime}}_{\in
W_{I}}m^{\prime\prime}\underbrace{q^{\prime\prime}}_{\in W_{J}}\in
W_{I}m^{\prime\prime}W_{J}$, so that $W_{I}mW_{J}=W_{I}m^{\prime\prime}W_{J}$
and therefore $m^{\prime\prime}\in W_{I}mW_{J}=W_{I}wW_{J}$. Since $m$ is the
Bruhat minimum of $W_{I}wW_{J}$, this entails $m\leq m^{\prime\prime}$. We
have thus found a circular chain of Bruhat inequalities:%
\[
m\leq m^{\prime\prime}\leq m^{\prime}\leq m.
\]
Thus, all $\leq$ signs must be equalities; in particular, $m^{\prime}=m$. So
\[
u=\underbrace{p^{\prime}}_{\in W_{I}}\underbrace{m^{\prime}}_{=m}%
\underbrace{q^{\prime}}_{\in W_{J}}\in W_{I}mW_{J}=W_{I}wW_{J}.
\]

So we have proved that each $u\in\left[  m,M\right]  $ satisfies $u\in
W_{I}wW_{J}$. That is, $\left[  m,M\right]  \subseteq W_{I}wW_{J}$. Combining
this with (\ref{pf.thm.double-interval.easy}), we conclude that $W_{I}%
wW_{J}=\left[  m,M\right]  $; in particular, $W_{I}wW_{J}$ is a Bruhat interval.
\end{proof}

\section{Proof of Theorem \ref{thm.triplemax}}

At last, we prove Theorem \ref{thm.triplemax}.

\begin{proof}
[Proof of Theorem \ref{thm.triplemax}]
First, since $W_{J}$ is finite, by Lemma \ref{lem-sandwich}, the set $uW_{J}v$
has a Bruhat maximum, which we will call $n$.

Next, consider the double coset $W_{I}nW_{K}$. By Lemma
\ref{double.thickening} \textbf{(a)} (applied to $P=I$, $Q=K$, $r=n$ and
$s=n$), the double coset $W_{I}nW_{K}$ has a Bruhat maximum, which we will
call $M$. We claim that this $M$ is the Bruhat maximum of $W_{I}uW_{J}vW_{K}$.

First note that $M\in W_{I}nW_{K}\subseteq W_{I}uW_{J}vW_{K}$, since $n \in
uW_{J}v$.

Next, let $p\in W_{I}uW_{J}vW_{K}$. Write $p$ as $p=xuyvz$ for some $x\in
W_{I}$, $y\in W_{J}$, and $z\in W_{K}$. Then, $uyv\in uW_{J}v$, so that
$uyv\leq n$ (since $n$ is the Bruhat maximum of $uW_{J}v$). Use this to apply
Lemma \ref{double.thickening} \textbf{(b)} with $r=uyv$, $s=n$, $P=I$, and
$Q=K$. Thus, we obtain
\[
p=xuyvz\leq M.
\]
Since this is true for all $p\in W_{I}uW_{J}vW_{K}$, we conclude that $M$ is
the Bruhat maximum of $W_{I}uW_{J}vW_{K}$ (since $M\in W_{I}uW_{J}vW_{K}$).
\end{proof}

\subsubsection*{Acknowledgments}

We thank Hankyung Ko for conversations on double cosets at ICERM (2025), and
Elizabeth Milicevic and Mark Skandera for further helpful discussions. Various
versions of GPT were used to generate parts of the proofs of Theorems
\ref{thm.double-interval} and \ref{thm.triplemax}. None of the writing is
AI-made, but everything was proofread by GPT (and by us as well).

\begin{noncompile}
Howard Hiller, \textit{Geometry of Coxeter groups}, Pitman Advanced Publishing
Program 1982.
\end{noncompile}

\end{document}